\documentclass[11pt, a4paper]{article}
\usepackage{amsmath,amssymb,latexsym,graphics,epsfig}
\usepackage{hyperref}
\usepackage{color}
\usepackage{amsthm}
\usepackage{graphicx,url}
\usepackage{amsopn}

\newtheorem{theo}{Theorem}

\newtheorem{corollari}[theo]{Corollary}
\newtheorem{propo}[theo]{Proposition}

\newtheorem{example}[theo]{Example}

\DeclareMathOperator{\dgr}{dgr}
\DeclareMathOperator{\tr}{tr}

\DeclareMathOperator{\spec}{sp}
\newcommand{\qqed}{\hfill$\Box$\medskip}

\def\vec0{\mbox{\boldmath $0$}}

\def\G{\Gamma}
\def\Re{\mathbb R}

\begin{document}
\title{The spectral excess theorem for distance-regular graphs having distance-$d$ graph with fewer distinct eigenvalues
}

\author{M.A. Fiol
\\ \\
{\small $^b$Universitat Polit\`ecnica de Catalunya, BarcelonaTech} \\
{\small Dept. de Matem\`atica Aplicada IV, Barcelona, Catalonia}\\
{\small (e-mail: {\tt
fiol@ma4.upc.edu})} \\
 }
\date{}
\maketitle

\begin{abstract}
Let $\G$ be a distance-regular graph with diameter $d$ and Kneser graph $K=\G_d$, the distance-$d$ graph of $\G$.
We say that $\G$ is partially antipodal when $K$ has fewer distinct eigenvalues than $\G$.
In particular, this is the case of antipodal distance-regular graphs ($K$ with only two distinct eigenvalues), and
the so-called half-antipodal distance-regular graphs ($K$ with only one negative eigenvalue). We provide a characterization of partially antipodal distance-regular graphs (among regular graphs with $d$ distinct eigenvalues) in terms of the spectrum and the mean number of vertices at maximal distance $d$ from every vertex. This can be seen as a general version of the so-called spectral excess theorem, which allows us to characterize those distance-regular graphs which are half-antipodal, antipodal, bipartite, or with Kneser graph being strongly regular.
\end{abstract}



\section{Preliminaries}
Let $\G$ be a distance-regular graph with  adjacency matrix $A$ and $d+1$ distinct eigenvalues. In the recent work of Brouwer and the author \cite{bf14}, we studied the situation where the distance-$d$ graph $\G_d$ of $\G$, or Kneser graph $K$, with adjacency matrix $A_d=p_d(A)$, has fewer distinct eigenvalues. In this case we say that $\G$ is {\em partially antipodal}. Examples are the so-called half antipodal ($K$ with only one negative eigenvalue, up to multiplicity), and antipodal distance-regular graphs ($K$ being disjoint copies of a complete graph). Here we generalize such a study to the case when $\G$ is a regular graph with $d+1$ distinct eigenvalues.
The main result of this paper is a characterization of partially antipodal distance-regular graphs, among regular graphs with $d+1$ distinct eigenvalues, in terms of the spectrum and the mean number of vertices at maximal distance $d$ from every vertex. This can be seen as a general version of the so-called spectral excess theorem, and allows us to characterize those distance-regular graphs which are half antipodal, antipodal, bipartite, or with Kneser graph being strongly regular. Other related characterizations of some of these cases were given by the author in \cite{f97,f00,f01}.
For background on distance-regular graphs and strongly regular graphs, we refer the reader to Brouwer, Cohen, and Neumaier \cite{bcn89}, Brouwer and Haemers \cite{bh12}, and Van Damm, Koolen and Tanaka \cite{dkt12}.

Let $\G$ be a regular (connected) graph with degree $k$, $n$ vertices, and spectrum $\spec \G=\{\lambda_0^{m_0},\lambda_1^{m_1},\ldots,\lambda_d^{m_d}\}$, where $\lambda_0(=k)>\lambda_1>\cdots>\lambda_d$, and $m_0=1$.
In this work, we use the following scalar product on the $(d+1)$-dimensional vector space of real polynomials modulo $m(x)=\prod_{i=0}^d (x-\lambda_i)$, that is, the minimal polynomial of $A$.
\begin{equation}
\label{prod}
\langle p,q\rangle_{\G}=\frac{1}{n}\tr (p(A)q(A))= \frac{1}{n}\sum_{i=0}^d m_i p(\lambda_i)q(\lambda_i), \qquad p,q\in \Re_{d}[x]/(m(x)).
\end{equation}
This is a special case of the inner product of symmetric $n\times n$ real matrices $M,N$, defined by $\langle M,N\rangle=\frac{1}{n}\tr(MN)$.
The {\it predistance polynomials} $p_0,p_1,\ldots, p_d$, introduced by the author and Garriga \cite{fg97}, are a sequence of orthogonal polynomials with respect to the  inner product \eqref{prod}, normalized in such a way that $\|p_i\|^2_{\G}=p_i(k)$.
Then, it is known that $\G$ is distance-regular if and only if such polynomials satisfy $p_i(A)=A_i$ (the adjacency matrix of the distance-$i$ graph $\G_i$) for $i=0,\ldots,d$, in which case they turn out to be the distance polynomials.
In fact, we have the following strongest proposition, which is a combination of results in \cite{fgy1b,ddfgg11}.
\begin{propo}
\label{propo-p(A)=Ad}
A regular graph $\G$ as above is distance-regular if and only if there exists a polynomial $p$ of degree $d$ such that $p(A)=A_d$, in which case $p=p_d$.\qqed
\end{propo}

Moreover, the Hoffman polynomial $H$, such that $H(\lambda_i)=n\delta_{0i}$ and $H(A)=J$, turns out to be $H=p_0+p_1+\cdots+p_d$. Also, as in the case of distance-regular graphs, the multiplicities of $\G$ can be obtained from the values of $p_d$ since,
\begin{equation}
\label{mult.vs.pd}
(-1)^i p_d(\lambda_i)\pi_i m_i= p_d(\lambda_0)\pi_0,\qquad i=1,\ldots,d.
\end{equation}
where $\pi_i=\prod_{j\neq i}|\lambda_i-\lambda_j|$.
Indeed, let $L_i(x)=\prod_{j\neq 0,i}(x-\lambda_j)/\prod_{j\neq 0,i}(\lambda_i-\lambda_j)$. Then, since $\dgr L_i=d-1$, \eqref{mult.vs.pd} follows from $\langle L_i,p_d\rangle_{\G}=0$ for $i=1,\ldots,d$.
Some interesting consequences of the above, together with other properties of the predistance polynomials are the following (for more details, see \cite{cffg09}):
 \begin{itemize}
\item
The values of $p_d$ at $\lambda_0,\lambda_1,\ldots,\lambda_d$ alternate in sign.
\item
Using the values of $p_d(\lambda_i)$, $i=0,\ldots,d$, given by  \eqref{mult.vs.pd}, in the equality $\|p_d\|_{\G}^2=p_d(\lambda_0)$, and solving for $p_d(\lambda_0)$ we get the so-called {\em spectral excess}
\begin{equation}
\label{spexcess}
p_d(\lambda_0)=n\left(\frac{\pi_0^2}{m_ip_i^2}\right)^{-1}.
\end{equation}
\item
For every $i=0,\ldots,d$, (any multiple of) the sum polynomial $q_i=p_0+\cdots+p_i$ maximizes the quotient $r(\lambda_0)/\|r\|_{\G}$ among the polynomials $r\in \Re_{i}[x]$ (notice that $q_i(\lambda_0)^2/\|q_i\|_{\G}^2=q_i(\lambda_0$)), and
$(1=)q_0(\lambda_0)<q_1(\lambda_0)<\cdots <q_d(\lambda_0)(=H(\lambda_0)=n)$.
\end{itemize}

Let $\G$ have $n$ vertices, $d+1$ distinct eigenvalues, and diameter $D(\le d)$. For $i=0,\ldots,D$, let $k_i(u)$ be the number of vertices at distance $i$ from vertex $u$. Let $s_i(u)=k_0(u)+\cdots+k_i(u)$. Of course, $s_0(u)=1$ and $s_D(u)=n$.
The following result can be seen as a version of the spectral excess theorem, due to Garriga and the author \cite{fg97} (for short proofs, see Van Dam \cite{vd08}, and Fiol, Gago and Garriga \cite{fgg10}):
\begin{theo}
\label{SPEtheorem}
Let $\G$ be a regular graph with spectrum $\spec \G=\{\lambda_0,\lambda_1^{m_1},\ldots,\lambda_d^{m_d}\}$, where $\lambda_0>\lambda_1>\cdots>\lambda_d$. Let $\overline{s_i}=\frac{1}{n}\sum_{u\in V} s_i(u)$ be the average number of vertices at
distance at most $i$ from every vertex in $\G$. Then, for any polynomial $r\in \Re_{d-1}[x]$ we have
\begin{equation}
\label{basic-ineq}
\frac{r(\lambda_0)^2}{\|r\|_{\G}^2}\le \overline{s_{d-1}},
\end{equation}
with equality if and only if $\G$ is distance-regular and $r$ is a nonzero multiple of $q_{d-1}$.
\end{theo}
\begin{proof}
Let $S_{d-1}=I+A+\cdots +A_{d-1}$.
As $\dgr r\le d-1$, $\langle r(A),J\rangle = \langle r(A), S_{d-1}\rangle$.
But $\langle r(A),J\rangle=\langle r,H\rangle_{\G}=r(\lambda_0)$. Thus, Cauchy-Schwarz inequality gives
$$
r^2(\lambda_0)\le \|r(A)\|^2\|S_{d-1}\|^2=\|r\|^2_{\G}\overline{s_{d-1}},
$$
whence \eqref{basic-ineq} follows.
Besides, in case of equality we have that $r(A)=\alpha S_{d-1}$ for some nonzero constant $\alpha$. Hence, the polynomial $p=H-(1/\alpha)r$ satisfies $p(A)=J-S_{d-1}=A_d$ and, from Proposition \ref{propo-p(A)=Ad},
$\G$ is distance-regular, $p=p_d$, and $r=\alpha q_{d-1}$. The converse in clear from $s_{d-1}=n-k_d=H(\lambda_0)-p_d(\lambda_0)=q_{d-1}(\lambda_0)$.
\end{proof}

In fact, as it was shown in \cite{f02}, the above result still holds if we change the arithmetic mean of the numbers $s_{d-1}(u)$, $u\in V$, by its harmonic mean.

\section{The results}
As commented above,
in  \cite{bf14} we studied
the situation where the distance-$d$ graph $\G_d$, of a distance-regular graph $\G$ with diameter $d$, has fewer distinct eigenvalues.
Now, we are interested in the case when $\G$ is regular and with $d+1$ distinct eigenvalues.
In this context, $p_d$ is the highest degree predistance polynomial and,
as $p_d(A)$ is not necessarily the distance-$d$ matrix $A_d$
(usually not even a $0$-$1$ matrix), we consider the distinct eigenvalues of $p_d(A)$ vs.
those of $A$. More precisely, given a set $H\subset \{0,\ldots,d\}$, we give conditions
for all $p_d(\lambda_i)$ with $i\in H$ taking the same value. Notice that, because the values
of $p_d$ at the mesh $\lambda_0,\lambda_1,\ldots,\lambda_d$ alternate in sign, the feasible sets $H$ must have either even or odd numbers

\subsection*{The case \boldmath $\lambda_0\notin H$}
We first study the more common case when $\lambda_0\notin H$.
For $i=1,\ldots,d$, let $\phi_i(x)=\prod_{j\neq 0,i}(x-\lambda_i)$, and consider again the Lagrange interpolating polynomial $L_i(x)=\phi_i(x)/\phi(\lambda_i)$, satisfying $L_i(\lambda_j)=\delta_{ij}$ for $j\neq 0$,
and $L_i(\lambda_0)=(-1)^{i+1}\frac{\pi_0}{\pi_i}$, where $\pi_i=|\phi_i(\lambda_i)|$.

\begin{theo}
\label{teo(basic)}
Let $\G$ be a regular graph with degree $k$, $n$ vertices, and spectrum $\spec \G=\{\lambda_0,\lambda_1^{m_1},\ldots,\lambda_d^{m_d}\}$, where $\lambda_0(=k)>\lambda_1>\cdots>\lambda_d$. Let $H\subset \{1,\ldots,d\}$.
For every $i=0,\ldots,d$, let $\pi_i=\prod_{j\neq i}|\lambda_i-\lambda_j|$. Let $\overline{k_d}=\frac{1}{n}\sum_{u\in V} k_d(u)$ be the average number of vertices at
distance $d$ from every vertex in $\G$. Then,
\begin{equation}
\label{ineq-basic-theo}
\overline{k_d}\le \frac{n\sum_{i\in H}m_i}
 {\left(\sum_{i\in H}\frac{\pi_0}{\pi_i}\right)^2+\sum_{i\notin H}\frac{\pi_0^2}{m_i\pi_i^2}\sum_{i\in H}m_i},
\end{equation}
and equality holds if and only if $\G$ is a distance-regular graph with constant $P_{id}=p_d(\lambda_i)$ for every $i\in H$.
\end{theo}

\begin{proof}
The clue is to apply Theorem \ref{SPEtheorem} with a polynomial $r\in \Re_{d-1}[x]$
having the desired properties of $q_{d-1}$. To this end, first notice  that, as $q_{d-1}=H-p_d$,
we have $q_{d-1}(\lambda_i)=-p_d(\lambda_i)$ for any $i\neq 0$.
Thus, we take the polynomial $r$ with values
$r(\lambda_i)=-t$ for $i\in H$, and
$r(\lambda_i)=-p_d(\lambda_i)$ for $i\notin H$, $i\neq 0$.
Then, using \eqref{mult.vs.pd},
\begin{align*}
r(x)&=-t\sum_{i\in H}L_i(x)-\sum_{i\notin H,i\neq 0}p_d(\lambda_i)L_i(x),\\
r(\lambda_0)&=-t\sum_{i\in H}(-1)^{i+1}\frac{\pi_0}{\pi_i}-\sum_{i\notin H,i\neq 0}p_d(\lambda_i)(-1)^{i+1}\frac{\pi_0}{\pi_i}\\
&=-t\sum_{i\in H}(-1)^{i+1}\frac{\pi_0}{\pi_i}-p_d(\lambda_0)^2\sum_{i\notin H,i\neq 0}\frac{\pi_0^2}{m_i\pi_i^2},\\
n\|r\|^2_{\G}
&=r(\lambda_0)^2+t^2\sum_{i\in H}m_i+\sum_{i\notin H, i\neq 0}m_i p_d(\lambda_i)^2.
\end{align*}
Thus, \eqref{basic-ineq} yields
\begin{equation}
\label{Phi(t)}
\Phi(t)=\frac{r(\lambda_0)^2}{\|r\|_{\G}^2}=\frac{n(\alpha t+\beta)^2}{(\alpha t+\beta)^2+\sigma t^2+\gamma}\le \overline{s_{d-1}}
\end{equation}
where
\begin{align}
\alpha & = \sum_{i\in H}(-1)^{i+1}\frac{\pi_0}{\pi_i},\qquad\qquad\qquad
\beta =-p_d(\lambda_0)\sum_{i\notin H,i\neq 0}\frac{\pi_0^2}{m_i\pi_i^2},
\label{alpha-beta}\\
\gamma & =\sum_{i\notin H,i\neq 0}m_ip_d(\lambda_i)^2=\sum_{i\notin H,i\neq 0}\frac{p_d(\lambda_0)^2}{m_i}\frac{\pi_0^2}{\pi_i^2}=-p_d(\lambda_0)\beta, \qquad\sigma=\sum_{i\in H}m_i.
\label{gamma-sigma}
\end{align}

Now, to have the best result in \eqref{Phi(t)} (and since we are mostly interested in the case of equality), we have to find the maximum of the function $\Phi(t)$, which is attained at $t_0=\alpha\gamma/\beta\sigma$. Then,
$$
\Phi_{\max}=\Phi(t_0)=\frac{n(\alpha^2\gamma+\beta^2\sigma)}{\alpha^2\gamma+\beta^2\sigma+\gamma\sigma}\le \overline{s_{d-1}}=n-\overline{k_d}.
$$
Thus, using \eqref{alpha-beta}-\eqref{gamma-sigma} and simplifying we get \eqref{ineq-basic-theo}.
In case of equality, we know, by Theorem \ref{SPEtheorem}, that $\G$ is distance-regular with $r(x)=\alpha q_{d-1}(x)$ for some constant $\alpha$. If $i\notin H,i\neq 0$, $r(\lambda_i)=-p_d(\lambda_i)=\alpha q_{d-1}(\lambda_i)=-\alpha p_d(\lambda_i)$, so that $\alpha=1$ since $p_d(\lambda_i)\neq 0$. Then, for every $i\in H$, we get
$$
P_{id}=p_d(\lambda_i)=H(\lambda_i)-q_{d-1}(\lambda_i)=-r(\lambda_i)=t_0.
$$
Conversely, if $\G$ is distance-regular, we have that $\overline{k_d}=k_d$, and, if $P_{id}$ is a constant, say, $\tau$  for every $i\in H$,
we obtain, from \eqref{mult.vs.pd}, that $\sigma=\frac{k_d}{p_d(\lambda_i)}\sum_{i\in H}(-1)^i\frac{\pi_0}{\pi_i}=\frac{k_d}{\tau}\alpha$, whence $\tau=\frac{1}{\sigma}k_d\alpha$.
Moreover,
$$
nk_d=\|p_d\|_{\G}^2=\sum_{i\notin H}m_i p_d(\lambda_i)^2+\sum_{i\in H}m_i \tau^2=
k_d^2\sum_{i\notin H}\frac{\pi_0^2}{m_i\pi_i^2}+k_d^2\frac{\left(\sum_{i\in H}\frac{\pi_0}{\pi_i}\right)^2}{\sum_{i\in H}m_i},
$$
and equality in \eqref{eq-spet-half-antipod} holds.
\end{proof}

As mentioned above, when $\G$ is already a distance-regular graph, Brouwer and the author \cite{bf14} gave parameter conditions for partial antipodality, and surveyed known examples. The different examples given here are withdrawn from such a paper.
\begin{example}
The Odd graph $O_5$, on $n=126$ vertices, has intersection array $\{5,4,4,3;1,1,2,2\}$, so that $k_d=60$, and spectrum
$5^1,3^{27},1^{42},-2^{48},-4^{8}$. Then, with $H=\{2,4\}$, the function $\Phi(t)$ is depicted in Fig. \ref{O5}.
Its  maximum is attained for $t_0=6$, and its value is $\Phi(6)=66=s_{d-1}$. Then, $P_{24}=P_{44}$.
\end{example}

\begin{figure}[ht]
  \center
  \includegraphics[scale=.4]{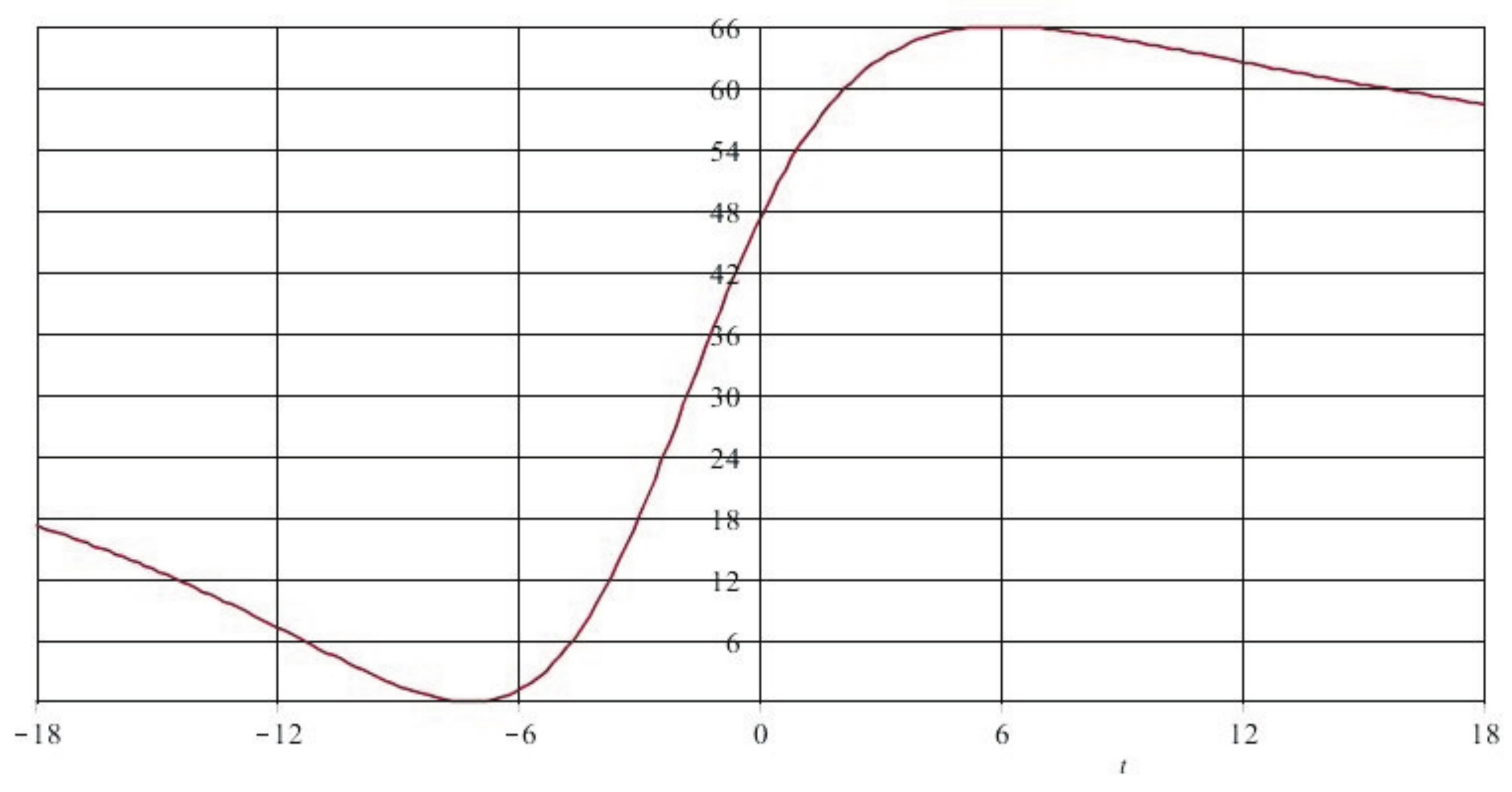}\\
  \caption{The function $\Phi(t)$ for $O_5$ with $H=\{2,4\}$.}
  \label{O5}
\end{figure}

Notice that if, in the above result, $H$ is a singleton, there is no restriction for the values of $p_d$, and then we get the so-called spectral excess theorem (originally proved by Garriga and the author \cite{fg97}).

\begin{corollari}[The spectral excess theorem]
Let $\G$ be a regular graph with spectrum $\spec \G$  and average number $\overline{k_d}$ as above.
Then $\G$ is distance-regular if and only if
$$
\overline{k_d}= p_d(\lambda_0)= n\left( \sum_{i=0}^d\frac{\pi_0^2}{m_i\pi_i^2}\right)^{-1}.
$$
\end{corollari}
\begin{proof}
Take $H=\{i\}$ for some $i\neq 0$ in Theorem \ref{teo(basic)}.
\end{proof}

As mentioned before, in \cite{bf14} a distance-regular graph $\G$ was said to be  half antipodal if the distance-$d$ graph has only one negative eigenvalue (i.e., $P_{id}$ is a constant for every $i=1,3,\ldots$).
Then, a direct consequence of Theorem \ref{teo(basic)} by taking $H=H_{\mbox{\scriptsize odd}}=\{1,3,\ldots\}$ is the following characterization of half antipodality.

\begin{corollari}
\label{teo(half-antipod)}
Let $\G$ be a regular graph as above. Then,
\begin{equation}
\label{eq-spet-half-antipod}
\overline{k_d}\le \frac{n\sum_{\mbox{\scriptsize $i$ odd}}m_i}
{\left(\sum_{\mbox{\scriptsize $i$ odd}}\frac{\pi_0}{\pi_i}\right)^2+\sum_{\mbox{\scriptsize $i$ even}}\frac{\pi_0^2}{m_i\pi_i^2}\sum_{\mbox{\scriptsize $i$ odd}}m_i},
\end{equation}
and equality holds if and only if $\G$ is a half antipodal distance-regular graph. \qqed
\end{corollari}

Recall that a regular graph is strongly regular if and only if it has at most three distinct eigenvalues (see e.g. \cite{g93}). Then, we can apply Theorem \ref{teo(basic)} with $H_{\mbox{\scriptsize even}}=\{2,4,\ldots \}$ and $H_{\mbox{\scriptsize odd}}=\{1,3,\ldots\}$ (and add up the two inequalities obtained) to obtain a characterization of those distance-regular graphs having strongly regular distance-$d$ graph.

\begin{corollari}
Let $\G$ be a regular graph as above. Then,
\begin{equation}
\label{eq-spet-half-antipod}
\overline{k_d}\le \frac{n^2}{\displaystyle
\left(\sum_{\mbox{\scriptsize $i$ even}}\frac{\pi_0}{\pi_i}\right)^2
+\left(\sum_{\mbox{\scriptsize $i$ odd}}\frac{\pi_0}{\pi_i}\right)^2
+\sum_{\mbox{\scriptsize $i$ even}}\frac{\pi_0^2}{m_i\pi_i^2}\sum_{\mbox{\scriptsize $i$ odd}}m_i
+\sum_{\mbox{\scriptsize $i$ odd}}\frac{\pi_0^2}{m_i\pi_i^2}\sum_{\mbox{\scriptsize $i$ even}}m_i},
\end{equation}
and equality holds if and only if $\G$ is a distance-regular graph with strongly regular distance-$d$ graph $\G_d$. \qqed
\end{corollari}

\begin{example}
The Wells graph, on $n=32$ vertices, has intersection array $\{5,4,1,1;1,1,4,5\}$ and spectrum
 $5^1,\sqrt{5}^{8},1^{10},-\sqrt{5}^{8},-3^{5}$. This graph is  $2$-antipodal, so that $k_d=1$. Then, Fig. \ref{Wells} shows the functions $\Phi_0(t)$ with $H_0=\{2,4\}$, and  $\Phi_1(t)$ with $H_1=\{1,3\}$. Their (common)  maximum value is attained for $t_0=1$ and $t_1=-1$, respectively, and  it is $\Phi_0(1)=\Phi_1(-1)=31=s_{d-1}$. Then, $P_{24}=P_{44}$ and $P_{14}=P_{34}$.
\end{example}

\begin{figure}[ht]
  \center
  \includegraphics[scale=.4]{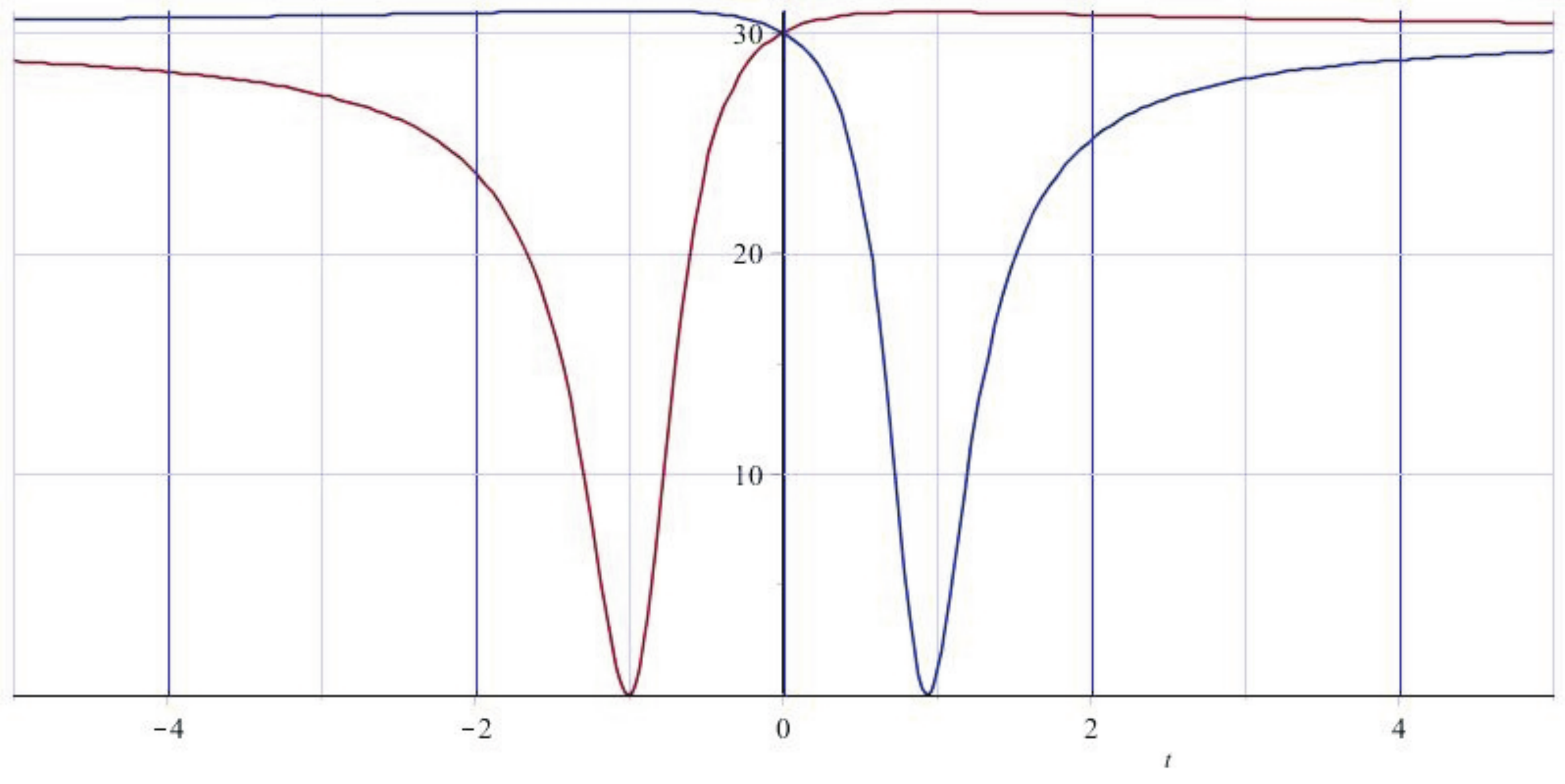}\\
  \caption{The functions $\Phi_0(t)$ (in red) with $H_0=\{2,4\}$, and $\Phi_1(t)$ (in blue) with $H_1=\{1,3\}$ of the Wells graph.}
  \label{Wells}
\end{figure}

In fact, the above expression can be simplified because 
$\sum_{\mbox{\scriptsize $i$ even}}m_i+\sum_{\mbox{\scriptsize $i$ odd}}m_i=n$,
 $\sum_{\mbox{\scriptsize $i$ even}}\frac{\pi_0}{\pi_i}=\sum_{\mbox{\scriptsize $i$ odd}}\frac{\pi_0}{\pi_i}$ (see \cite{f00}),
 and, from \eqref{spexcess}, $\sum_{\mbox{\scriptsize $i$ even}}\frac{\pi_0^2}{m_i\pi_i^2}+\sum_{\mbox{\scriptsize $i$ odd}}\frac{\pi_0^2}{m_i\pi_i^2}=n/p_d(\lambda_0)$.
 Anyway, we have written \eqref{eq-spet-half-antipod} as it is to emphasize the `symmetries' between even and odd terms.

The following result was used in \cite{bf14,f01} for the case of distance-regular graphs (where $p_d(\lambda_i)=P_{id}$).

\begin{corollari}
Let $\G$ be a regular graph with eigenvalues $\lambda_0>\lambda_1>\cdots >\lambda_d$. Let $H\subset \{1,\ldots,d\}$.
Then, $p_d(\lambda_i)=p_d(\lambda_j)$ for every $i,j\in H$ if and only if $\sum_{i\neq j}(m_i\pi_i-m_j\pi_j)^2=0$.
\end{corollari}
\begin{proof}
Notice that the right hand side of \eqref{ineq-basic-theo} is just the spectral excess $p_d(\lambda_0)$, which is given by \eqref{spexcess}. Then, the result follows from from equating both expressions and simplifying.
\end{proof}

\subsection*{The case \boldmath $\lambda_0\in H$}

To deal with this case, we could proceed as above by defining conveniently a degree $d-1$ polynomial $r$. Then the proof is similar to the one for Theorem \ref{teo(basic)}.
If $\lambda_0\in H$ then  $p(\lambda_i)=p(\lambda_0)$ for any $i\in H$.
Moreover, the odd indexes, cannot belong to $H$. In particular $1\notin H$. For instance, a possible choice for $r\in \Re_{d-1}[x]$ is:
\begin{itemize}
\item
$r(\lambda_0)=n-p_d(\lambda_0)$,\ $r(\lambda_i)=-p_d(\lambda_0)$ for $i\in H$, $i\neq 0$.
\item
$r(\lambda_i)=-tp_d(\lambda_i)$ for $i\notin H$, $i\neq 1$,
\end{itemize}

Hovewer, we can follow a more direct approach by using \eqref{Phi(t)}. First, the following result was proved in \cite{bf14}:

\begin{propo}[\mbox{\cite[Prop. 8]{bf14}}]
\label{0ddd}
Let $\G$ be a distance regular graph with diameter $d$.
If $P_{0d} = P_{id}$ then $i$ is even.
Let $i>0$ be even. Then $P_{0d} = P_{id}$ if and only
$\Gamma$ is antipodal, or $i=d$ and $\Gamma$ is bipartite. \qqed
\end{propo}


%

\begin{theo}
Let $\G$ be a regular graph with $n$ vertices, spectrum $\spec \G$ as above, and mean excess $\overline{k_d}$. Then, for every $i=1,\ldots,d$,
\begin{equation}
\label{ineq-antipod-theo}
\overline{k_d}\le \frac{n\left(m_i+\sum_{j\neq 0,i}\frac{\pi_0^2}{m_j\pi_j^2}\right)}
 {\left(\frac{\pi_0}{\pi_i}+\sum_{j\neq 0,i}\frac{\pi_0^2}{m_j\pi_j^2}\right)^2+m_i+\sum_{j\neq 0,i}\frac{\pi_0^2}{m_j\pi_j^2}}.
\end{equation}
Moreover:
\begin{itemize}
\item[$(a)$]
Equality holds for some $i\neq d$ if and only it holds for any $i=1,\ldots,d$ and  $\G$ is an antipodal distance-regular graph.
\item[$(b)$]
Equality holds only for $i=d$ if and only if $\G$ is a bipartite, but not antipodal, distance-regular graph.
\end{itemize}
\end{theo}
\begin{proof}
The inequality \eqref{ineq-antipod-theo} follows from \eqref{Phi(t)} by taking $H=\{i\}$ for some even $i\neq 0$, and choosing  $t=p_d(\lambda_0)$. Then, in case of equality,
Theorem \ref{teo(basic)} tells us that $\G$ is distance-regular. Then, $\G_d$ is a regular graph with equal eigenvalues $P_{0d}$ and $P_{id}$. So, the result follows
from Proposition \ref{0ddd}.
\end{proof}

\begin{example}
For the Wells graph the right hand expression of \eqref{ineq-antipod-theo} gives $1(=k_4)$ for any $i=1,\ldots,4$, in concordance with its antipodal character.
In contrast, the folded $10$-cube $FQ_{10}$, on $n=512$
vertices, has intersection array $\{10,9,8,7,6;1,2,3,4,10\}$ and spectrum
 $10^1,6^{45},2^{210},-2^{210},-6^{45},-10^1$. Then, the right hand expression of \eqref{ineq-antipod-theo} gives $234.16$, $293.36$, $293.36$, $234.16$ for $i=1,2,3,4$, respectively, and $126(=k_5)$ for $i=5$, showing that $FQ_{10}$ is a bipartite distance-regular graph, but not antipodal.
\end{example}

Another characterization of antipodal distance-regular graphs was given by the author in \cite{f97} by assuming that the distance $d$-graph of a regular graph is already antipodal.

\medskip

\section*{Acknowledgments} 
Research supported by the
{\em Ministerio de Ciencia e Innovaci\'on}, Spain, and the {\em European Regional Development Fund} under project MTM2011-28800-C02-01. 



\begin{thebibliography}{99}

%


\bibitem{bcn89}
A.E. Brouwer, A.M. Cohen, and A. Neumaier, \emph{Distance-Regular Graphs},
Springer-Verlag, Berlin-New York, 1989.

\bibitem{bf14}
A.E. Brouwer and M.A. Fiol, Distance-regular graphs where the distance-$d$ graph
has fewer distinct eigenvalues, preprint (2014); arXiv:1409.0389 [math.CO].

\bibitem{bh12}
A.E. Brouwer and W.H. Haemers, \emph{Spectra of Graphs},
Springer, 2012; available online at \url{http://homepages.cwi.nl/~aeb/math/ipm/}.

\bibitem{cffg09}
M. C\'amara, J. F\`abrega, M.A. Fiol, and E. Garriga,
Some families of orthogonal polynomials of a discrete variable and
their applications to graphs and codes, {\em Electron. J. Combin.} {\bf 16(1)} (2009), \#R83.



%
%
\bibitem{vd08}
E.R. van Dam, The spectral excess theorem for distance-regular
graphs: a global (over)view, {\em Electron. J. Combin.} {\bf 15(1)}
(2008), \#R129.

%
%
%
\bibitem{dkt12}
E.R. van Dam, J.H. Koolen, and H. Tanaka, Distance-regular
graphs, manuscript (2014), available online at
\url{https://sites.google.com/site/edwinrvandam/home/papers/drg.pdf}.

\bibitem{ddfgg11}
C. Dalf\'o, E.R. van Dam, M.A. Fiol, E. Garriga, and B.L. Gorissen, On almost distance-regular graphs, {\em J. Combin. Theory, Ser. A} {\bf 118} (2011) 1094--1113.

\bibitem{f97}
M.A. Fiol, An eigenvalue characterization of antipodal distance-regular graphs, {\em Electron. J. Combin.} {\bf 4} (1997), \#R30.

\bibitem{f00}
M.A. Fiol,
A quasi-spectral characterization of strongly distance-regular graphs, {\em Electron. J. Combin.} {\bf 7} (2000), \#R51.

\bibitem{f01}
M.A. Fiol,
Some spectral characterization of strongly distance-regular graphs, {\em Combin. Probab. Comput.} {\bf 10} (2001), no. 2, 127--135.

\bibitem{f02}
M.A. Fiol,
Algebraic characterizations of distance-regular graphs, \emph{Discrete
Math.} {\bf 246} (2002) 111--129.

%
\bibitem{fgg10}
M.A. Fiol, S. Gago, and E. Garriga,
A simple proof of the spectral excess theorem for distance-regular graphs,
{\em Linear Algebra Appl.} {\bf 432} (2010), 2418--2422.

\bibitem{fg97}
M.A. Fiol and E. Garriga,
From local adjacency polynomials to locally pseudo-distance-regular graphs,
\emph{J. Combin. Theory Ser. B} {\bf 71} (1997), 162--183.
%
%
\bibitem{fgy1b}
M.A. Fiol, E. Garriga, and J.L.A. Yebra,
Locally pseudo-distance-regular graphs, {\it J.  Combin. Theory Ser. B}
{\bf 68} (1996), 179--205.

\bibitem{g93}
C.D. Godsil, {\it Algebraic Combinatorics}, Chapman and Hall, NewYork, 1993.

%
%

\end{thebibliography}
\end{document}